\documentclass[10pt,a4paper]{article}

\usepackage{amsmath, amsthm, amssymb, amsfonts}
\usepackage[T1]{fontenc}
\usepackage{graphicx}

\numberwithin{equation}{section}
\newtheorem{theorem}{Theorem}[section]
\newtheorem{lemma}[theorem]{Lemma}
\newtheorem{definition}[theorem]{Definition}
\newtheorem{corollary}[theorem]{Corollary}

\newtheorem{example}[theorem]{Example}

\DeclareSymbolFont{AMSb}{U}{msb}{m}{n}
\DeclareMathSymbol{\N}{\mathalpha}{AMSb}{"4E}
\DeclareMathSymbol{\R}{\mathalpha}{AMSb}{"52}
\DeclareMathSymbol{\Z}{\mathalpha}{AMSb}{"5A}
\DeclareMathSymbol{\D}{\mathalpha}{AMSb}{"44}
\DeclareMathSymbol{\X}{\mathalpha}{AMSb}{"58}
\DeclareMathSymbol{\s}{\mathalpha}{AMSb}{"53}
\newcommand{\M}{\mathsf{M}}
\newcommand{\m}{\mathsf{m}}
\newcommand{\W}{\mathcal{P}}

\newcommand{\V}{\mathsf{vol}}

\begin{document}

\title{Ricci Bounds for Euclidean and Spherical Cones}

\author{Kathrin Bacher, Karl-Theodor Sturm}

\date{}

\maketitle

\begin{abstract}
We prove generalized lower Ricci bounds for Euclidean and spherical cones over compact Riemannian manifolds. These cones are regarded as complete metric measure spaces. We show that the Euclidean cone over an $n$-dimensional Riemannian manifold whose Ricci curvature is bounded from below by $n-1$ satisfies the curvature-dimension condition $\mathsf{CD}(0,n+1)$ and that the spherical cone over the same manifold fulfills the curvature-dimension condition $\mathsf{CD}(n,n+1)$.
\end{abstract}

\section{Introduction}

In two similar but independent approaches, Sturm \cite{sa,sb} and Lott \& Villani \cite{lva,lv} presented a concept of generalized lower Ricci curvature bounds for metric measure spaces $(\M,\mathsf{d},\m)$.
The full strength of this concept appears if the condition $\mathsf{Ric}(\M,\mathsf{d},\m)\ge K$ is combined with a kind of upper bound $N$ on the dimension.
This leads to the so-called curvature-dimension condition $\mathsf{CD}(K,N)$ which can be formulated in terms of optimal transportation for each pair of numbers $K\in\mathbb{R}$ and $N\in[1,\infty)$.

A complete Riemannian manifold satisfies $\mathsf{CD}(K,N)$ if and only if its Ricci curvature is bounded from below by $K$ and its dimension from above by $N$.

A broad variety of geometric and functional analytic results can be deduced from the curvature-dimension condition $\mathsf{CD}(K,N)$. Among them are the Brunn-Minkowski inequality and the theorems by Bishop-Gromov, Bonnet-Myers and Lichnerowicz. Moreover, the condition $\mathsf{CD}(K,N)$ is stable under convergence with respect to the $\mathsf{L}_2$-transportation distance $\D$.

\subsection{Statement of the Main Results}

Let $\M$ be a compact $n$-dimensional Riemannian manifold (with Riemannian distance $\mathsf{d}$ and Riemannian volume $\mathsf{vol}$) satisfying $\mathsf{diam}(\M)\leq\pi$. The \textit{Euclidean cone} $\mathsf{Con}(\M)$ over $\M$ is defined as the quotient of the product $\M\times [0,\infty)$ obtained by identifying all points in the fiber $\M\times\{0\}$. This point is called the origin $\mathsf{O}$ of the cone. It is equipped with a metric $\mathsf{d_{Con}}$ defined by the cosine formula
$$\mathsf{d_{Con}}((x,s),(y,t))=\sqrt{s^2+t^2-2st\cos(\mathsf{d}(x,y))},$$
and with a measure $\nu$ defined as the product $d\nu(x,s):=d\mathsf{vol}(x)\otimes s^Nds$.

\begin{theorem} \label{Cheeger}
The Ricci curvature of $\M$ is bounded from below by $n-1$ if and only if the metric measure space $(\mathsf{Con}(\M),\mathsf{d_{Con}},\nu)$ satisfies the curvature-dimension condition $\mathsf{CD}(0,n+1)$.
\end{theorem}

The heuristic interpretation is that the Euclidean cone -- regarded as a metric measure space -- has non-negative Ricci curvature in a generalized sense. Note that already in 1982, Cheeger and Taylor \cite{cha,chb} were able to prove that the punctured Euclidean cone $\mathsf{Con}(\M)\setminus\{\mathsf{O}\}$ constructed over a compact $n$-dimensional Riemannian manifold $\M$ with $\mathsf{Ric}\geq n-1$ is an $(n+1)$-dimensional Riemannian manifold with $\mathsf{Ric}\geq 0$. Of course, $\mathsf{Con}(\M)\setminus\{\mathsf{O}\}$ is not a complete manifold and in general, $\mathsf{Con}(\M)$ on its own is not a smooth one. In particular, the Ricci curvature in the classical sense is not defined in its singularity $\mathsf{O}$.\\
Metric cones play an important role in the study of limits of Riemannian manifolds. Assume for instance that $(\M,\mathsf{d})$ is the Gromov-Hausdorff limit of a sequence of complete $n$-dimensional Riemannian manifolds whose Ricci curvature is uniformly bounded from below. Then in the non-collapsed case, every tangent cone $\mathsf{T}_x\M$ is a metric cone $\mathsf{Con}(\mathsf{S}_x\M)$ with $\mathsf{diam}(\mathsf{S}_x\M)\leq\pi$ \cite{cca,ccb}. The latter we would expect from the diameter estimate by Bonnet-Myers if $\mathsf{Ric}_{\mathsf{S}_x\M}\geq n-2$ which in turn is consistent with the formal assertion \textquoteleft$\mathsf{Ric}_{\mathsf{T}_x\M}\geq 0$\textquoteright.

\bigskip

As a second main result we deduce a generalized lower Ricci bound for the \textit{spherical cone} $\Sigma(\M)$ over the compact Riemannian manifold $\M$. It is defined as the quotient of the product space $\M\times[0,\pi]$ obtained by contracting all points in the fiber $\M\times\{0\}$ to the south pole $\mathcal{S}$ and all points in the fiber $\M\times\{\pi\}$ to the north pole $\mathcal{N}$. It is endowed with a metric $\mathsf{d}_\Sigma$ defined via
$$\cos\left(\mathsf{d}_\Sigma(p,q)\right)=\cos s\cos t+\sin s\sin t\cos\left(\mathsf{d}(x,y)\right)$$
for $p=(x,s), q=(y,t)\in\Sigma(\M)$ and with a measure $d\nu(x,s):=d\mathsf{vol}(x)\otimes(\sin^Ns ds)$.

\begin{theorem}
The Ricci curvature of $\M$ is bounded from below by $n-1$ if and only if the metric measure space $(\Sigma(\M),\mathsf{d_\Sigma},\nu)$ satisfies the curvature-dimension condition $\mathsf{CD}(n,n+1)$.
\end{theorem}

\subsection{Basic Definitions and Notations}

Throughout this paper,  $(\mathsf{M},\mathsf{d},\mathsf{m})$ denotes a \textit{metric measure space} consisting of a complete separable metric space $(\mathsf{M},\mathsf{d})$ and a locally finite measure $\mathsf{m}$ on $(\mathsf{M},\mathcal{B}(\mathsf{M}))$, that is, the volume $\mathsf{m}(B_r(x))$ of balls centered at $x$ is finite for all $x\in \mathsf{M}$ and all sufficiently small $r>0$. The metric space $(\mathsf{M},\mathsf{d})$  is called a \textit{length space} if and only if $\mathsf{d}(x,y)=\inf\mathsf{Length}(\gamma)$ for all $x,y\in\mathsf{M}$, where the infimum runs over all curves $\gamma$ in $\mathsf{M}$ connecting $x$ and $y$. $(\mathsf{M},\mathsf{d})$ is called a \textit{geodesic space} if and only if every two points $x,y\in\mathsf{M}$ are connected by a curve $\gamma$ with $\mathsf{d}(x,y)=\mathsf{Length}(\gamma)$. Such a curve is called \textit{geodesic}.

A \textit{non-branching} metric measure space $(\mathsf{M},\mathsf{d},\mathsf{m})$ consists of a geodesic metric space $(\mathsf{M},\mathsf{d})$ such that for every tuple $(z,x_0,x_1,x_2)$ of points in $\mathsf{M}$ for which $z$ is a midpoint of $x_0$ and $x_1$ as well as of $x_0$ and $x_2$, it follows that $x_1=x_2$.

The \textit{diameter} $\mathsf{diam}(\mathsf{M},\mathsf{d},\mathsf{m})$ of a metric measure space $(\mathsf{M},\mathsf{d},\mathsf{m})$ is defined as the diameter of its support, namely, $\mathsf{diam}(\mathsf{M},\mathsf{d},\mathsf{m}):=\sup\{\mathsf{d}(x,y): x,y\in\mathsf{supp}(\mathsf{m})\}$.

$(\mathcal{P}_2(\mathsf{M},\mathsf{d}),\mathsf{d}_{\mathsf{W}})$ denotes the \textit{$\mathsf{L}_2$-Wasserstein space} of probability measures $\nu$ on $(\mathsf{M},\mathcal{B}(\mathsf{M}))$ with finite second moments which means that $\int_\mathsf{M}\mathsf{d}^2(x_0,x)d\nu(x)<\infty$
for some (hence all) $x_0\in \mathsf{M}$. The \textit{$\mathsf{L}_2$-Wasserstein distance} $\mathsf{d}_{\mathsf{W}}(\mu,\nu)$ between two probability measures
$\mu,\nu\in\mathcal{P}_2(\mathsf{M},\mathsf{d})$ is defined as
$$\mathsf{d}_{\mathsf{W}}(\mu,\nu)=\inf\left\{\left(\int_{\mathsf{M}\times \mathsf{M}}\mathsf{d}^2(x,y)d\mathsf{q}(x,y)\right)^{1/2}:
\text{$\mathsf{q}$ coupling of $\mu$ and $\nu$}\right\}.$$
Here the infimum ranges over all \textit{couplings} of $\mu$ and $\nu$ which are probability measures on $\mathsf{M}\times \mathsf{M}$ with marginals $\mu$ and $\nu$. $(\mathcal{P}_2(\M,\mathsf{d}),\mathsf{d_W})$ is a complete separable metric space. The subspace of $\mathsf{m}$-absolutely continuous measures is denoted by $\mathcal{P}_2(\mathsf{M},\mathsf{d},\mathsf{m})$.

For general  $K\in\mathbb{R}$ and $N\in[1,\infty)$ the condition $\mathsf{CD}(K,N)$ states that for each pair
$\nu_0,\nu_1\in\mathcal{P}_2(\mathsf{M},\mathsf{d},\mathsf{m})$  there exist  an optimal
coupling $\mathsf{q}$ of $\nu_0=\rho_0\mathsf{m}$ and $\nu_1=\rho_1\mathsf{m}$ and a geodesic $\nu_t=\rho_t\, m$ in $\mathcal{P}_2(\mathsf{M},\mathsf{d},\mathsf{m})$ connecting them such that
\begin{equation} \label{CD}
\begin{split}
\mathsf{S}_{N'}(\nu_t|\mathsf{m})&:=-\int_\mathsf{M}\rho_t^{1-1/N'}d\m\\
&\leq-\int_{\mathsf{M}\times
\mathsf{M}}\left[\tau^{(1-t)}_{K,N'}(\mathsf{d}(x_0,x_1))\rho^{-1/N'}_0(x_0)+
\tau^{(t)}_{K,N'}(\mathsf{d}(x_0,x_1))\rho^{-1/N'}_1(x_1)\right]d\mathsf{q}(x_0,x_1)
\end{split}
\end{equation}
for all $t\in [0,1]$ and all $N'\geq N$.
In order to define the \textit{volume distortion coefficients} $\tau^{(t)}_{K,N}(\cdot)$ we introduce for $\theta\in\R_+$,
\begin{equation*}
\mathfrak{S}_k(\theta):=
\begin{cases}
\frac{\sin(\sqrt{k}\theta)}{\sqrt{k}\theta}& \text{if $k>0$}\\
1& \text{if $k=0$}\\
\frac{\sinh(\sqrt{-k}\theta)}{\sqrt{-k}\theta}& \text{if $k<0$}
\end{cases}
\end{equation*}
and set for $t\in[0,1]$,
\begin{equation*}
\sigma^{(t)}_{K,N}(\theta):=
\begin{cases}
\infty& \text{if $K\theta^2\geq N\pi^2$}\\
t\frac{\mathfrak{S}_{K/N}(t\theta)}{\mathfrak{S}_{K/N}(\theta)}& \text{else}
\end{cases}
\end{equation*}
as well as $\tau^{(t)}_{K,N}(\theta):=t^{1/N}\sigma^{(t)}_{K,N-1}(\theta)^{1-1/N}$. By replacing the volume distortion coefficients $\tau^{(t)} _{K,N}(\cdot)$ by the slightly smaller coefficients $\sigma^{(t)} _{K,N}(\cdot)$ in the definition of $\mathsf{CD}(K,N)$ the reduced curvature-dimension condition $\mathsf{CD}^*(K,N)$ is obtained. This condition was introduced and studied in \cite{bs}.

The definitions of the condition $\mathsf{CD}(K,N)$ in \cite{sb} and \cite{lva} slightly differ. We follow the notation of \cite{sb}. For non-branching spaces, both concepts coincide.
In this case, it suffices to verify (\ref{CD}) for $N'=N$ since this already implies (\ref{CD}) for all $N'\ge N$.

\subsection{Some Technical Ingredients}

We recall two technical statements we will refer to in the course of this note. According to \cite[Lemma 2.11]{sa}, the following lemma holds true:

\begin{lemma} \label{optimaltransport}
\begin{itemize}
\item[(i)] For each pair $\mu,\nu\in\mathcal{P}_2(\M,\mathsf{d})$ there exists a coupling $\mathsf{q}$ - called {\rm optimal coupling} - such that
$$\mathsf{d^2_W}(\mu,\nu)=\int_{\M\times \M}\mathsf{d^2}(x,y) \ d\mathsf{q}(x,y).$$
\item[(ii)]For each geodesic $\Gamma:[0,1]\rightarrow\mathcal{P}_2(\M,\mathsf{d})$, each $k\in\N$ and each partition
$$0=t_0<t_1<\dots<t_k=1$$
there exists a probability measure $\mathsf{\hat q}$ on $\M^{k+1}$ with the following properties:
\begin{itemize}
\item[$\ast$]The projection on the $i$-th factor is $\Gamma(t_i)$ for all $i=0,1,\dots,k$.
\item[$\ast$]For $\mathsf{\hat q}$-almost every $x=(x_0,\dots,x_k)\in\M^{k+1}$ and every $i,j=0,1,\dots,k$,
$$\mathsf{d}(x_i,x_j)=|t_i-t_j|\mathsf{d}(x_0,x_k).$$
In particular, for every pair $i,j\in\{0,1,\dots,k\}$ the projection on the $i$-th and $j$-th factor is an optimal coupling of $\Gamma(t_i)$ and $\Gamma(t_j)$.
\item[(iii)]If $(\M,\mathsf{d})$ is a non-branching space, then we have in the situation of (ii) for $\mathsf{\hat q}$-almost every $(x_0,x_1,x_2)$ and $(y_0,y_1,y_2)$ in $\M^3$,
$$x_1=y_1 \Rightarrow (x_0,x_2)=(y_0,y_2).$$
\end{itemize}
\end{itemize}
\end{lemma}

In this framework, the notion of \textit{cyclical monotonicity} plays an important role in the sense of Lemma \ref{opttrans} taken from \cite[Theorem 5.10]{vib}:

\begin{definition}[Cyclical monotonicity]
Let $(\M,\mathsf{d})$ be a metric space. A subset $\Xi\subset\M\times\M$ is called $\mathsf{d}^2$-cyclically monotone if and only if for any $k\in\N$ and for any family $(x_1,y_1),\dots,(x_k,y_k)$ of points in $\Xi$ the inequality
$$\sum^k_{i=1}\mathsf{d}^2(x_i,y_i)\leq\sum^k_{i=1}\mathsf{d}^2(x_i,y_{i+1})$$
holds with the convention $y_{k+1}=y_1$.
\end{definition}

\begin{lemma}[Optimal transference plan] \label{opttrans}
The optimal coupling $\mathsf{q}$ of two probability measures $\nu_0,\nu_1\in\W_2(\M,\mathsf{d},\m)$ is concentrated on a $\mathsf{d}^2$-cyclically monotone set.
\end{lemma}

\section{Euclidean Cones over Metric Measure Spaces} \label{cones}

\subsection{$N$-Euclidean Cones over Metric Measure Spaces} \label{euclcones}

\begin{definition}[$N$-Euclidean cone]
For a metric measure space $(\mathsf{M},\mathsf{d},\mathsf{m})$ with $\mathsf{diam}(\M)\leq\pi$ and $N\geq1$, the \textit{$N$-Euclidean cone} $(\mathsf{Con}(\M),\mathsf{d_{Con}},\nu)$ is a metric measure space defined as follows:
\begin{itemize}
\item[$\diamond$] $\mathsf{Con}(\M):=\mathsf{M}\times[0,\infty)/ \mathsf{M}\times\{0\}$
\item[$\diamond$] For $(x,s),(x',t)\in \mathsf{M}\times[0,\infty)$
$$\mathsf{d_{Con}}((x,s),(x',t)):=\sqrt{s^2+t^2-2st\cos\left(\mathsf{d}(x,x')\right)}$$
\item[$\diamond$] $d\nu(x,s):=d\mathsf{m}(x)\otimes s^Nds$.
\end{itemize}
\end{definition}

\begin{example}
The most prominent example in this setting is the unit sphere $\s^n\subseteq\R^{n+1}$ endowed with the spherical angular metric $\mathsf{d}_\measuredangle$ that is, the distance between two points in $\s^n$ is given by the Euclidean angle between them, and with the Lebesgue measure restricted to $\s^n$. To construct the $n$-Euclidean cone over $\s^n$, we draw a ray from the origin $0$ in $\R^{n+1}$ through every point $x\in \s^n$. A point $a\in\mathsf{Con}(\s^n)$ can be described by a pair $(x,t)$ where $x$ is a point in $\s^n$ belonging to the ray $0a$ and $t=|a|$ is the Euclidean distance from the origin. By this construction we obtain the whole Euclidean space $\R^{n+1}$.
\vspace{-0.8cm}
\begin{center}
\hspace{-2cm}
\includegraphics[height=5cm]{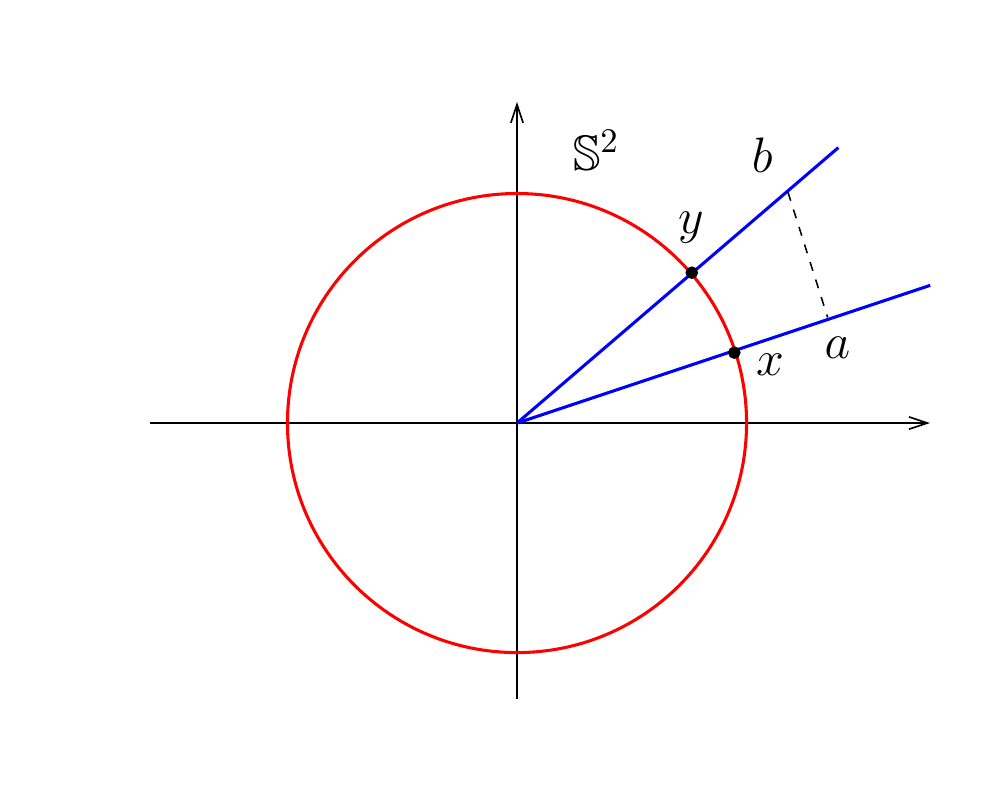}
\end{center}
\vspace{-0.5cm}
The Euclidean distance $|a-b|$ between two points $a=(x,t)$ and $b=(y,s)$ in $\mathsf{Con}(\s^n)$ is given in terms of the angular metric $\mathsf{d}_\measuredangle$ and the lengths $t$ and $s$ via the cosine formula
$$|a-b|=\sqrt{t^2+s^2-2ts\cos(\mathsf{d}_\measuredangle(x,y))}.$$
Thus, the definition of the metric $\mathsf{d_{Con}}$ and the measure $\nu$ ensures that the $n$-Euclidean cone over $\s^n$ is the Euclidean space $\R^{n+1}$ equipped with the Euclidean metric and the Lebesgue measure expressed in spherical coordinates.
\end{example}

\subsection{Optimal Transport on Euclidean Cones} \label{optimaltransportoncone}

Let $(\M,\mathsf{d},\mathsf{m})$ be a metric measure space with full support $\M=\mathsf{supp}(\m)$ satisfying the curvature-dimension condition $\mathsf{CD}(N-1,N)$ for some $N\geq 1$. The diameter estimate by Bonnet-Myers implies that $\mathsf{diam}(\M)\leq\pi$. We denote by $(\mathsf{Con}(\M),\mathsf{d_{Con}},\nu)$ the $N$-Euclidean cone over $(\M,\mathsf{d},\mathsf{m})$. For each pair of probability measures $\mu_0$ and $\mu_1$ in $\W_2(\mathsf{Con}(\M),\mathsf{d_{Con}},\nu)$ there exists a geodesic $\Gamma:[0,1]\to\W_2(\mathsf{Con}(\M),\mathsf{d_{Con}})$ connecting them. The probability measures $\Gamma(t)$ are not necessarily absolutely continuous for all $t$ because $\W_2(\mathsf{Con}(\M),\mathsf{d_{Con}},\nu)$ is not necessarily a geodesic space. Thus, theoretically, it might happen that all mass is transported from $\mu_0$ to $\mu_1$ through the origin.
\vspace{-0.8cm}
\begin{center}
\hspace{-2cm}
\includegraphics[height=4.5cm]{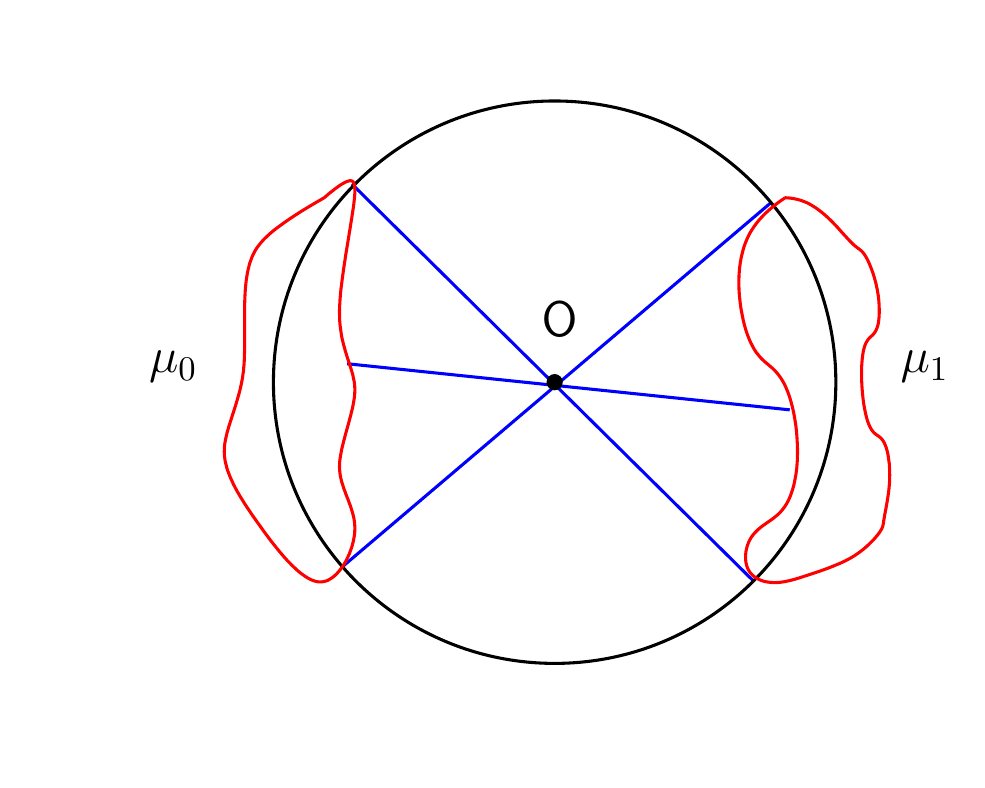}
\end{center}
\vspace{-0.5cm}
But due to Theorem \ref{origin}, this phenomenon does not occur. We consider the partition $0=t_0<t_{1/2}=\tfrac{1}{2}<t_1=1$ of $[0,1]$. Due to Lemma \ref{optimaltransport}, there exists a probability measure $\mathsf{\hat q}$ on $\mathsf{Con}(\M)^3$ with the following properties:
\begin{itemize}
\item[$\ast$]The projection on the $i$-th factor is $\Gamma(t_i)$ for all $i=0,\tfrac{1}{2},1$.
\item[$\ast$]For $\mathsf{\hat q}$-almost every $x=(x_0,x_{1/2},x_1)\in\mathsf{Con}(\M)^3$, the point $x_{1/2}$ is a midpoint of $x_0$ and $x_1$.
In particular, the projection on the $i$-th and $j$-th factor is an optimal coupling of $\Gamma(t_i)$ and $\Gamma(t_j)$ for $i,j=0,\tfrac{1}{2},1$.
\end{itemize}

In the sequel we use the notation $\mathsf{O}:=\M\times\{0\}\in\mathsf{Con}(\M)$. The following theorem states that the optimal transport from $\mu_0$ to $\mu_1$ does not touch the origin.

\begin{theorem} \label{origin}
It holds that
$$\mathsf{\hat q}\left(\{(x_0,x_{1/2},x_1)\in\mathsf{Con}(\M)^3:x_{1/2}=\mathsf{O}\}\right)=0.$$
\end{theorem}

\begin{proof}
This proof is divided into three parts: Each part starts with the formulation of a lemma.
\begin{lemma} \label{anti}
Let two points $x_0=(\phi_0,r)$ and $x_1=(\phi_1,s)$ in $\mathsf{Con}(\M)$ be given and let $\gamma:[0,1]\to\mathsf{Con}(\M)$ be a geodesic connecting them, meaning $\gamma(0)=x_0$ and $\gamma(1)=x_1$. If $\gamma_{1/2}:=\gamma(\tfrac{1}{2})=\mathsf{O}$, then $\phi_0$ and $\phi_1$ are antipodes in $\M$ in the sense that $\mathsf{d}(\phi_0,\phi_1)=\pi$.
\end{lemma}
\begin{proof}[Proof of Lemma \ref{anti}] Due to the definition of $\mathsf{d_{Con}}$, we have first of all
$$r=\mathsf{d_{Con}}(x_0,\gamma_{1/2})=\tfrac{1}{2}\mathsf{d_{Con}}(x_0,x_1)=\mathsf{d_{Con}}(\gamma_{1/2},x_1)=s,$$
and secondly,
\begin{align*}
r^2=\mathsf{d_{Con}}(x_0,\gamma_{1/2})^2&=\frac{1}{4}\mathsf{d_{Con}}(x_0,x_1)^2\\
&=\tfrac{1}{4}\left[2r^2-2r^2\cos(\mathsf{d}(\phi_0,\phi_1))\right]=\tfrac{1}{2}\left[1-\cos(\mathsf{d}(\phi_0,\phi_1))\right]r^2\\
\end{align*}
which implies that
$$\cos(\mathsf{d}(\phi_0,\phi_1))=-1.$$
Due to \cite[Corollary 2.6]{sb}, the generalized Bonnet-Myers theorem on diameter bounds of metric measure spaces yields
$$\mathsf{d}(\phi_0,\phi_1)\leq\pi.$$
Therefore, we conclude that $\mathsf{d}(\phi_0,\phi_1)=\pi$.
\end{proof}
\begin{lemma}[Ohta] \label{ohta}
The set $\mathsf{S}(\phi,\pi):=\{\phi_a\in\M:\mathsf{d}(\phi,\phi_a)=\pi\}$ of antipodes of $\phi$ consists of at most one point for every $\phi\in\M$.
\end{lemma}
For a proof of Lemma \ref{ohta}, we refer to \cite[Theorem 4.5]{oa}.
\begin{lemma} \label{uniq}
Either $\{(x_0,x_{1/2},x_1)\in\mathsf{supp}(\mathsf{\hat q}):x_{1/2}=\mathsf{O}\}$ is the empty set or it coincides with $\{(\mathsf{O},\mathsf{O},\mathsf{O})\}$ or there exists at most one pair $(\phi_0,\phi_1)$ of antipodes in $\M$ with the following property: If $(\mathsf{O},\mathsf{O},\mathsf{O})\not =a=(a_0,a_{1/2},a_1)\in\mathsf{supp}(\mathsf{\hat q})\subseteq\mathsf{Con}(\M)^3$ satisfies $a_{1/2}=\mathsf{O}$ then $a_0=(\phi_0,r)$ and $a_1=(\phi_1,r)$ for some $r\in(0,\infty)$.
\end{lemma}
\begin{proof}[Proof of Lemma \ref{uniq}]
We assume that there are two different pairs $(\phi_0,\phi_1)$ and $(\varphi_0,\varphi_1)$ of antipodes in $\M$ such that there exist $a=(a_0,a_{1/2},a_1),b=(b_0,b_{1/2},b_1)\in\mathsf{supp}(\mathsf{\hat q})$ fulfilling $a_{1/2}=\mathsf{O}=b_{1/2}$ as well as $a_i=(\phi_i,r)$ and $b_i=(\varphi_i,s)$ for $i=0,1$ and some $r,s\in(0,\infty)$. We denote by $\mathsf{q}$ the projection of $\mathsf{\hat q}$ on the first and third factor, formally
$$\mathsf{q}:=\left(\mathsf{p_{01}}\right)_\ast\mathsf{\hat q},$$
where
\begin{align*}
\mathsf{p_{01}}:\mathsf{Con}(\M)^3&\rightarrow\mathsf{Con}(\M)^2\\
(x_0,x_{1/2},x_1)&\mapsto(x_0,x_1).
\end{align*}
Then $\mathsf{q}$ is an optimal coupling of $\mu_0$ and $\mu_1$:
$$\mathsf{\hat d_W}^2(\mu_0,\mu_1)=\int_{\mathsf{Con}(\M)\times\mathsf{Con}(\M)}\mathsf{d_{Con}}(x_0,x_1)^2d\mathsf{q}(x_0,x_1).$$
Lemma \ref{anti} and Lemma \ref{ohta} imply
\begin{align*}
\mathsf{d}^2_{\mathsf{Con}}&(a_0,b_1)+\mathsf{d}^2_{\mathsf{Con}}(b_0,a_1)\\
&=\left[r^2+s^2-2rs\cos\underset{<\pi}{\underbrace{\left(\mathsf{d}(\phi_0,\varphi_1)\right)}}\right] + \left[r^2+s^2-2rs\cos\underset{<\pi}{\underbrace{\left(\mathsf{d}(\varphi_0,\phi_1)\right)}}\right]\\
&<2(r+s)^2\leq4r^2+4s^2=\mathsf{d}^2_{\mathsf{Con}}(a_0,a_1)+\mathsf{d}^2_{\mathsf{Con}}(b_0,b_1).
\end{align*}
This contradicts the fact that the support of $\mathsf{q}$ is $\mathsf{d}^2_{\mathsf{Con}}$-cyclically monotone due to Theorem \ref{opttrans}.
\end{proof}
\
\\
Lemma \ref{uniq} finishes the proof of Theorem \ref{origin}.
\end{proof}

\subsection{Application to Riemannian Manifolds. I} \label{cheeger}

We consider a compact and complete $n$-dimensional Riemannian manifold $(\mathsf{M},\mathsf{d}, \V)$ denoting by $\mathsf{d}$ the Riemannian distance and by $\V$ the Riemannian volume.

\begin{theorem}
The $n$-Euclidean cone $(\mathsf{Con}(\M),\mathsf{d_{Con}},\nu)$ over a compact, complete $n$-dimensional Riemannian manifold $(\mathsf{M},\mathsf{d}, \V)$ satisfies $\mathsf{CD}(0,n+1)$ if and only if $\M$ fulfills $\mathsf{Ric}_\M\geq n-1$.
\end{theorem}

\begin{proof}
We consider two measures $\mu_0,\mu_1\in\mathcal{P}_2(\mathsf{Con}(\M),\mathsf{d_{Con}},\nu)$. Then there exists a geodesic $(\mu_t)_{t\in[0,1]}$ in $\W_2(\mathsf{Con}(\M),\mathsf{d_{Con}})$ connecting $\mu_0$ and $\mu_1$. As above, we consider the partition
$$0=t_0<t_{1/2}=\tfrac{1}{2}<t_1=1$$
of $[0,1]$ and a probability measure $\mathsf{\hat q}$ on $\mathsf{Con}(\M)^3$ satisfying the appropriate properties of Lemma \ref{optimaltransport}. For $\varepsilon>0$ we denote by $\mathsf{\hat q_\varepsilon}$ the restriction of $\mathsf{\hat q}$ to $\mathsf{Con}(\M)^3_\varepsilon:=\left(\mathsf{Con}(\M)\setminus B_\varepsilon(\mathsf{O})\right)^3$, meaning that
$$\mathsf{\hat q_\varepsilon}(A)=\frac{1}{\mathsf{\hat q}(\mathsf{Con}(\M)^3_\varepsilon)}\mathsf{\hat q}(A\cap\mathsf{Con}(\M)^3_\varepsilon)$$
for $A\subseteq\mathsf{Con}(\M)^3$. Furthermore, we define $\mu^\varepsilon_i$ as the projection of $\mathsf{\hat q_\varepsilon}$ on the $i$-th factor
$$\mu^\varepsilon_i:=(\mathsf{p_i})_* \mathsf{\hat q_\varepsilon}$$
where
\begin{align*}
\mathsf{p_i}:\mathsf{Con}(\M)^3&\rightarrow\mathsf{Con}(\M)\\
(x_0,x_{1/2},x_1)&\mapsto x_i
\end{align*}
for $i=0,\tfrac{1}{2},1$, and $\mathsf{q_\varepsilon}:=\left(\mathsf{p_{01}}\right)_\ast\mathsf{\hat q_\varepsilon}$
where
\begin{align*}
\mathsf{p_{01}}:\mathsf{Con}(\M)^3&\rightarrow\mathsf{Con}(\M)^2\\
(x_0,x_{1/2},x_1)&\mapsto(x_0,x_1).
\end{align*}
Then for every $\varepsilon>0$, $\mathsf{q_\varepsilon}$ is an optimal coupling of $\mu^\varepsilon_0$ and $\mu^\varepsilon_1$ and $\mu^\varepsilon_{1/2}$ is a midpoint of them. We derive from Theorem \ref{origin} that the following convergence statements hold true,
$$\mathsf{q_\varepsilon}(B)\underset{\varepsilon\to 0}{\rightarrow}\mathsf{q}(B) \mbox{ \ and \ } \mu^\varepsilon_i(C)\underset{\varepsilon\to 0}{\rightarrow}\mu_i(C)$$
for $i=0,\tfrac{1}{2},1$, $B\subseteq\mathsf{Con}(\M)^2$ and $C\subseteq\mathsf{Con}(\M)$, respectively, where $\mathsf{q}:=\left(\mathsf{p_{01}}\right)_\ast\mathsf{\hat q}$.

\smallskip
\
\\
A more than $20$-year old result by Cheeger and Taylor \cite{cha, chb} is the basis of our proof. For simplicity we introduce the notation $\mathsf{C}_0:=\mathsf{Con}(\M)\setminus\{\mathsf{O}\}$.
\begin{lemma}[Cheeger/Taylor]
The punctured Euclidean cone $\mathsf{C}_0$ is an $(n+1)$-dimensional Riemannian manifold. For $(\phi,r)\in\mathsf{C}_0$ with $\phi\in\M$ and $r>0$ the tangent space $\mathsf{T}_{(\phi,r)}\mathsf{C}_0$ can be parametrized as $\mathsf{T}_\phi\M\oplus\{\lambda\tfrac{\partial}{\partial r}:\lambda\in\R\}$.
Moreover, for $(v,\lambda)\in\mathsf{T}_{(\phi,r)}\mathsf{C}_0$ with $v\in\mathsf{T}_\phi\M$ and $\lambda\in\R$ the identity
$$\mathsf{Ric}_{\mathsf{C}_0}(v+\lambda\tfrac{\partial}{\partial r},v+\lambda\tfrac{\partial}{\partial r})=
\mathsf{Ric}_\M(v,v)-(n-1)\parallel v\parallel^2_{\mathsf{T}_\phi\M}$$
holds true. In particular, $\mathsf{Ric}_\M\geq n-1$ if and only if $\mathsf{Ric}_{\mathsf{C}_0}\geq 0$.
\end{lemma}
\
\\
For fixed $\varepsilon>0$ we embed $\mathsf{Con}(\M)\setminus B_\varepsilon(\mathsf{O})$ in a complete Riemannian manifold $\tilde\M_\varepsilon$ whose Ricci curvature is bounded from below:

\begin{center}
\includegraphics[height=5.0cm]{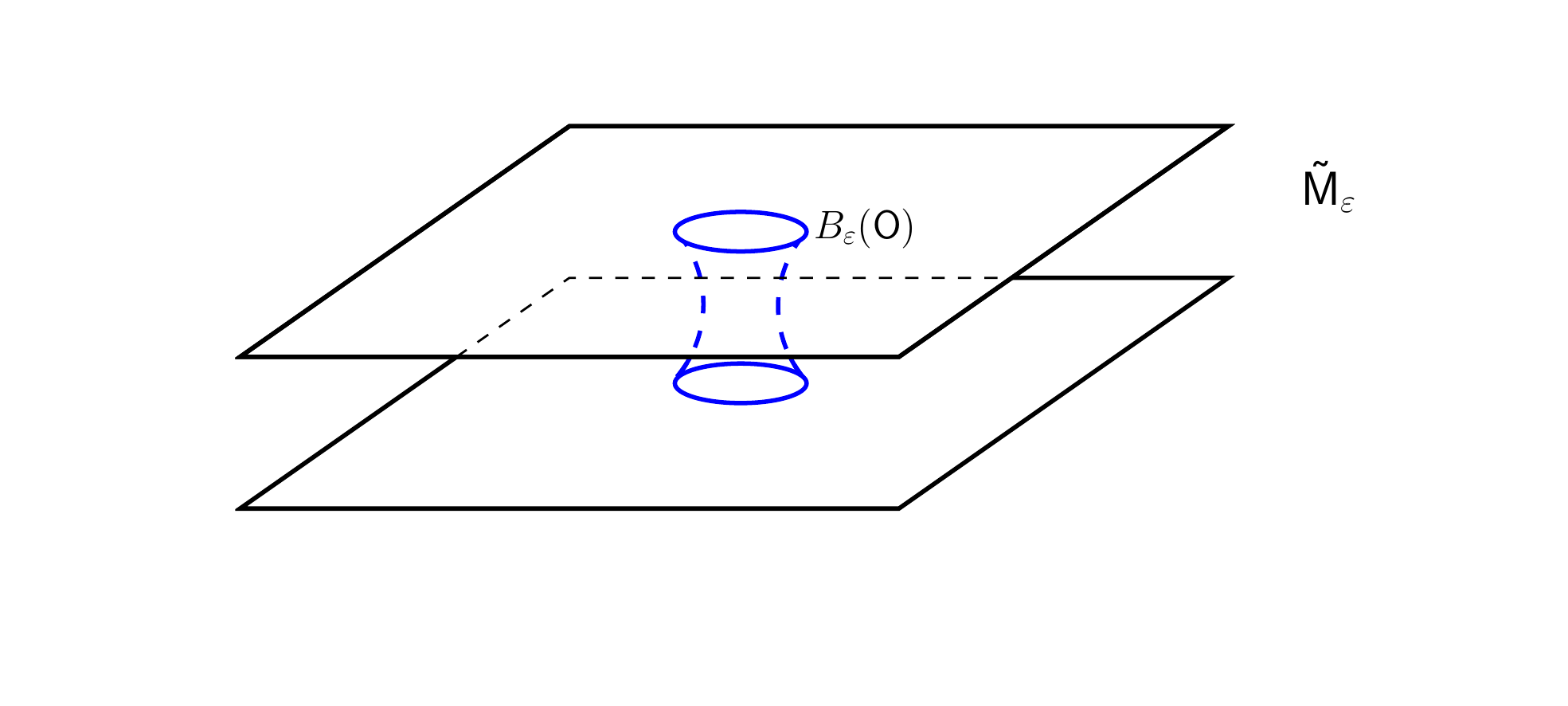}
\end{center}

The inclusion $\mathsf{Con}(\M)\setminus B_\varepsilon(\mathsf{O})\subseteq\tilde\M_\varepsilon$ in a complete Riemannian manifold implies that $\mu^\varepsilon_{1/2}$ is the unique midpoint of $\mu^\varepsilon_0$ and $\mu^\varepsilon_1$ and satisfies
$$\mathsf{S}_{n'}(\mu^\varepsilon_{1/2}|\nu)
\leq\tfrac{1}{2}\mathsf{S}_{n'}(\mu^\varepsilon_0|\nu)+\tfrac{1}{2}\mathsf{S}_{n'}(\mu^\varepsilon_1|\nu)$$
for all $\varepsilon>0$. Passing to the limit $\varepsilon\to 0$ yields according to the convergence statements,
$$\mathsf{S}_{n'}(\mu_{1/2}|\nu)\leq\tfrac{1}{2}\mathsf{S}_{n'}(\mu_0|\nu)+\tfrac{1}{2}\mathsf{S}_{n'}(\mu_1|\nu)$$
for all $n'\geq n+1$.

\end{proof}

\section{Spherical Cones over Metric Measure Spaces} \label{spherical}

\subsection{$N$-Spherical Cones over Metric Measure Spaces}

There are further objects with famous Euclidean ancestors -- among them is the \textit{spherical cone} or \textit{suspension} over a topological space $\M$. We begin with a familiar example: In order to construct the Euclidean sphere $\s^{n+1}$ out of its equator $\s^n$ we add two poles $\mathcal{S}$ and $\mathcal{N}$ and connect them via semicircles, the \textit{meridians}, through every point in $\s^n$.

In the general case of abstract spaces $\M$, we consider the product $\M\times I$ of $\M$ and a segment $I=[0,a]$ and contract each of the fibers $\M\times\{0\}$ and $\M\times\{a\}$ to a point, the \textit{south} and the \textit{north pole}, respectively. The resulting space is denoted by $\Sigma(\M)$ and is called the spherical cone over $\M$.

If $(\M,\mathsf{d})$ is a length space with $\mathsf{diam}(\M)\leq\pi$, we choose $a=\pi$ and define the \textit{spherical cone metric} $\mathsf{d}_\Sigma$ on $\Sigma(\M)$ by the formula
$$\cos\left(\mathsf{d}_\Sigma(p,q)\right)=\cos s\cos t+\sin s\sin t\cos\left(\mathsf{d}(x,y)\right)$$
for $p=(x,s), q=(y,t)\in\Sigma(\M)$.

\begin{definition}[$N$-spherical cone]

The \textit{$N$-spherical cone} $(\Sigma(\M),\mathsf{d}_\Sigma,\nu)$ over a metric measure space $(\mathsf{M},\mathsf{d},\mathsf{m})$ satisfying $\mathsf{diam}(\M)\leq\pi$ is a metric measure space defined as follows:
\begin{itemize}
\item[$\diamond$] $\Sigma(\M):=\mathsf{M}\times[0,\pi]\Big/ \mathsf{M}\times\{0\},\M\times\{\pi\}$
\item[$\diamond$] For $(x,s),(x',t)\in \mathsf{M}\times[0,\pi]$
$$\cos\left(\mathsf{d}_\Sigma((x,s),(x',t))\right):=\cos s\cos t+\sin s\sin t\cos\left(\mathsf{d}(x,x')\right)$$
\item[$\diamond$] $d\nu(x,s):=d\m(x)\otimes(\sin^Ns ds)$.
\end{itemize}
\end{definition}

For a nice introduction and detailed information about Euclidean and spherical cones over metric spaces we refer to \cite{bi}.

\subsection{Optimal Transport on Spherical Cones} \label{optimaltransportonsphericalcone}

This section is structured in the same manner as the corresponding section devoted to optimal transport on Euclidean cones. Again we consider a metric measure space $(\M,\mathsf{d},\mathsf{m})$ with full support $\M=\mathsf{supp}(\m)$ satisfying the curvature-dimension condition $\mathsf{CD}(N-1,N)$ for some $N\geq 1$.

We denote by $(\Sigma(\M),\mathsf{d}_\Sigma,\nu)$ the $N$-spherical cone over $(\M,\mathsf{d},\mathsf{m})$ with poles $\mathcal{S}:=\M\times\{0\}$ and $\mathcal{N}:=\M\times\{\pi\}$. For each pair of probability measures $\mu_0$ and $\mu_1$ in $\W_2(\Sigma(\M),\mathsf{d}_\Sigma,\nu)$ there exists a geodesic $\Gamma:[0,1]\to\W_2(\Sigma(\M),\mathsf{d}_\Sigma)$ connecting them. The critical case in this situation would be if all mass was transported from $\mu_0$ to $\mu_1$ through the poles. But Theorem \ref{poles} excludes this scenario.

We fix $0<s<1$ and consider the partition $0=t_0<t_s=s<t_1=1$ of $[0,1]$. Due to Lemma \ref{optimaltransport}, there exists a probability measure $\mathsf{\tilde q}$ on $\Sigma(\M)^3$ with properties listed in Section \ref{optimaltransportoncone} - with the only difference that in the current situation the time point $\tfrac{1}{2}$ is replaced by $s$.

The following theorem states that the optimal transport from $\mu_0$ to $\mu_1$ does not involve the poles.

\begin{theorem} \label{poles}
It holds that
$$\mathsf{\tilde q}\left(\{(x_0,x_s,x_1)\in\Sigma(\M)^3:x_s=\mathcal{S} \ \mbox{or} \ x_s=\mathcal{N}\}\right)=0.$$
\end{theorem}

\begin{proof}
We restrict our attention to the proof of the statement
\begin{equation} \label{statement}
\mathsf{\tilde q}\left(\{(x_0,x_s,x_1)\in\Sigma(\M)^3:x_s=\mathcal{S}\}\right)=0.
\end{equation}
Analogous calculations lead to the complete statement of Theorem \ref{poles}. The proof of (\ref{statement}) consists of two steps:
\begin{lemma} \label{antispher}
Let two points $x_0=(\phi_0,r)$ and $x_1=(\phi_1,t)$ in $\Sigma(\M)$ be given and let $\gamma:[0,1]\to\Sigma(\M)$ be a geodesic connecting them. If $\gamma_s:=\gamma(s)=\mathcal{S}$, then $\phi_0$ and $\phi_1$ are antipodes in $\M$.
\end{lemma}
\begin{proof}[Proof of Lemma \ref{antispher}] Due to the definition of $\mathsf{d}_\Sigma$, it holds that
$$r=\mathsf{d}_\Sigma(x_0,\gamma_s)=s\mathsf{d}_\Sigma(x_0,x_1)$$
as well as
$$t=\mathsf{d}_\Sigma(\gamma_s,x_1)=(1-s)\mathsf{d}_\Sigma(x_0,x_1)$$
and consequently, $t=\tfrac{1-s}{s}r$. Inserting this equality in the expression for $\cos\left(\tfrac{r}{s}\right)$ we obtain
\begin{align*}
\cos\left(\tfrac{r}{s}\right)&=\cos\left(\mathsf{d}_\Sigma(x_0,x_1)\right)\\
&=\cos r\cos\left(\tfrac{1-s}{s}r\right)+\sin r\sin\left(\tfrac{1-s}{s}r\right)\cos\left(\mathsf{d}(\phi_0,\phi_1)\right).
\end{align*}
This leads to
\begin{align*}
\cos(\mathsf{d}(\phi_0,\phi_1))&=\frac{\cos\left(\tfrac{r}{s}\right)-\cos r\cos\left(\tfrac{1-s}{s}r\right)}{\sin r\sin\left(\tfrac{1-s}{s}r\right)}\\
&=\frac{\cos\left(\tfrac{r}{s}\right)-\tfrac{1}{2}\left[\cos\left(\tfrac{2s-1}{s}r\right)+\cos\left(\tfrac{r}{s}\right)\right]}
{\tfrac{1}{2}\left[\cos\left(\tfrac{2s-1}{s}r\right)-\cos\left(\tfrac{r}{s}\right)\right]}\\
&=\frac{\tfrac{1}{2}\left[\cos\left(\tfrac{r}{s}\right)-\cos\left(\tfrac{2s-1}{s}r\right)\right]}
{\tfrac{1}{2}\left[\cos\left(\tfrac{2s-1}{s}r\right)-\cos\left(\tfrac{r}{s}\right)\right]}=-1.
\end{align*}
Finally, we deduce from $\mathsf{d}(\phi_0,\phi_1)\leq\pi$ that $\mathsf{d}(\phi_0,\phi_1)=\pi$.
\end{proof}
\begin{lemma} \label{uniqspher}
Either $\{(x_0,x_s,x_1)\in\mathsf{supp}(\mathsf{\tilde q}):x_s=\mathcal{S}\}$ is the empty set or it coincides with $\{(\mathcal{S},\mathcal{S},\mathcal{S})\}$ or there exists at most one pair $(\phi_0,\phi_1)$ of antipodes in $\M$ with the following property: If $(\mathcal{S},\mathcal{S},\mathcal{S})\not=a=(a_0,a_s,a_1)\in\mathsf{supp}(\mathsf{\tilde q})\subseteq\Sigma(\M)^3$ satisfies $a_s=\mathcal{S}$ then $a_0=(\phi_0,r)$ and $a_1=(\phi_1,\tfrac{1-s}{s}r)$ for some $r\in(0,\pi)$.
\end{lemma}
\begin{proof}[Proof of Lemma \ref{uniqspher}]
We assume that there are two different pairs $(\phi_0,\phi_1)$ and $(\varphi_0,\varphi_1)$ of antipodes in $\M$ such that there exist $a=(a_0,a_s,a_1),b=(b_0,b_s,b_1)\in\mathsf{supp}(\mathsf{\tilde q})$ fulfilling $a_s=\mathcal{S}=b_s$ as well as $a_0=(\phi_0,r)$, $a_1=(\phi_1,\tfrac{1-s}{s}r)$ and $b_0=(\varphi_0,t)$, $b_1=(\varphi_1,\tfrac{1-s}{s}t)$ for $i=0,1$ and some $r,t\in(0,\pi)$. Lemma \ref{antispher} and Lemma \ref{ohta} imply
\begin{align*}
\mathsf{d}^2_\Sigma&(a_0,b_1)+\mathsf{d}^2_\Sigma(b_0,a_1)\\
&=\arccos^2\left[\cos r\cos\left(\tfrac{1-s}{s}t\right)+\sin r\sin\left(\tfrac{1-s}{s}t\right)\cos\underset{<\pi}{\underbrace{\left(\mathsf{d}(\phi_0,\varphi_1)\right)}}\right] +\\
&\hspace{3cm}+ \arccos^2\left[\cos\left(\tfrac{1-s}{s}r\right)\cos t+\sin\left(\tfrac{1-s}{s}r\right)\sin t\cos\underset{<\pi}{\underbrace{\left(\mathsf{d}(\varphi_0,\phi_1)\right)}}\right]\\
&<\arccos^2\left[\cos r\cos\left(\tfrac{1-s}{s}t\right)-\sin r\sin\left(\tfrac{1-s}{s}t\right)\right]\\
&\hspace{3cm}+\arccos^2\left[\cos\left(\tfrac{1-s}{s}r\right)\cos t-\sin\left(\tfrac{1-s}{s}r\right)\sin t\right]\\
&=\arccos^2\left[\cos\left(r+\tfrac{1-s}{s}t\right)\right]+\arccos^2\left[\cos\left(\tfrac{1-s}{s}r+t\right)\right]\\
&=\left[r+\tfrac{1-s}{s}t\right]^2+\left[\tfrac{1-s}{s}r+t\right]^2\\
&=\left[r^2+\left(\tfrac{1-s}{s}\right)^2r^2\right]+\left[t^2+\left(\tfrac{1-s}{s}\right)^2t^2\right]+4\tfrac{1-s}{s}rt\\
&\leq\left[r^2+\left(\tfrac{1-s}{s}\right)^2r^2+2\tfrac{1-s}{s}r^2\right]+\left[t^2+\left(\tfrac{1-s}{s}\right)^2t^2+2\tfrac{1-s}{s}t^2\right]\\
&=\left[\tfrac{r}{s}\right]^2+\left[\tfrac{t}{s}\right]^2=\mathsf{d}^2_\Sigma(a_0,a_1)+\mathsf{d}^2_\Sigma(b_0,b_1).
\end{align*}
This contradicts the fact that the support of $\mathsf{q}:=\left(\mathsf{p_{01}}\right)_\ast\mathsf{\tilde q}$ being an optimal coupling of $\mu_0$ and $\mu_1$ is $\mathsf{d}^2_\Sigma$-cyclically monotone where $\mathsf{p_{01}}:\Sigma(\M)^3\rightarrow\Sigma(\M)^2,(x_0,x_s,x_1)\mapsto(x_0,x_1)$.
\end{proof}
\
\\
At the end of the second step, Theorem \ref{poles} is proved.
\end{proof}

\subsection{Application to Riemannian Manifolds. II} \label{reference}

\begin{theorem} \label{sphconriemman}
The $n$-spherical cone $(\Sigma(\M),\mathsf{d}_\Sigma,\nu)$ over a compact and complete $n$-dimensional Riemannian manifold $(\mathsf{M},\mathsf{d}, \V)$ satisfies $\mathsf{CD}(n,n+1)$ if and only if $\M$ fulfills $\mathsf{Ric}_\M\geq n-1$.
\end{theorem}

\begin{proof}
We consider two measures $\mu_0,\mu_1\in\mathcal{P}_2(\Sigma(\M),\mathsf{d}_\Sigma,\nu)$. Then there exists a geodesic $(\mu_t)_{t\in[0,1]}$ in $\W_2(\Sigma(\M),\mathsf{d}_\Sigma)$ connecting $\mu_0$ and $\mu_1$. As before, we consider for a fixed but arbitrary $0<s<1$ the partition
$$0=t_0<t_s=s<t_1=1$$
of $[0,1]$ and a probability measure $\mathsf{\tilde q}$ on $\Sigma(\M)^3$ satisfying the appropriate properties of Lemma \ref{optimaltransport}. For $\varepsilon>0$ we denote by $\mathsf{\tilde q_\varepsilon}$ the restriction of $\mathsf{\tilde q}$ to $\Sigma(\M)^3_\varepsilon:=\left[\Sigma(\M)\setminus(B_\varepsilon(\mathcal{S})\cup B_\varepsilon(\mathcal{N}))\right]^3$, meaning that
$$\mathsf{\tilde q_\varepsilon}(A)=\frac{1}{\mathsf{\tilde q}(\Sigma(\M)^3_\varepsilon)}\mathsf{\tilde q}(A\cap\Sigma(\M)^3_\varepsilon)$$
for $A\subseteq\Sigma(\M)^3$. Furthermore, we define $\mu^\varepsilon_i$ as the projection of $\mathsf{\tilde q_\varepsilon}$ on the $i$-th factor
$$\mu^\varepsilon_i:=(\mathsf{p_i})_* \mathsf{\tilde q_\varepsilon}$$
for $i=0,s,1$ and $\mathsf{q_\varepsilon}$ as the projection of $\mathsf{\tilde q_\varepsilon}$ on the first and third factor $$\mathsf{q_\varepsilon}:=\left(\mathsf{p_{01}}\right)_\ast\mathsf{\tilde q_\varepsilon}.$$
Then for every $\varepsilon>0$, $\mathsf{q_\varepsilon}$ is an optimal coupling of $\mu^\varepsilon_0$ and $\mu^\varepsilon_1$ and $\mu^\varepsilon_s$ is an $s$-intermediate point of them. We derive from Theorem \ref{poles} that the following convergence statements hold true,
$$\mathsf{q_\varepsilon}(B)\underset{\varepsilon\to 0}{\rightarrow}\mathsf{q}(B) \mbox{ \ and \ } \mu^\varepsilon_i(C)\underset{\varepsilon\to 0}{\rightarrow}\mu_i(C)$$
for $i=0,s,1$, $B\subseteq\Sigma(\M)^2$ and $C\subseteq\Sigma(\M)$, respectively, where $\mathsf{q}:=\left(\mathsf{p_{01}}\right)_\ast\mathsf{\tilde q}$.

\smallskip
\
\\
The core of our proof is shown by Petean \cite{jp}. We use the notation $\Sigma_0:=\Sigma(\M)\setminus\{\mathcal{S},\mathcal{N}\}$.

\begin{lemma}[Petean]
The punctured spherical cone $\Sigma_0$ is an incomplete $(n+1)$-dimensional Riemannian manifold whose tangent space $\mathsf{T}_{(\phi,r)}\Sigma_0$ at $(\phi,r)\in\Sigma_0$ with $\phi\in\M$ and $0<r<\pi$ can be parametrized as
$$\mathsf{T}_{(\phi,r)}\Sigma_0=\mathsf{T}_\phi\M\oplus\{\lambda\tfrac{\partial}{\partial r}:\lambda\in\R\}$$
and whose metric tensor is given by
$$\parallel v+\lambda\tfrac{\partial}{\partial r}\parallel^2_{\mathsf{T}_{(\phi,r)}\Sigma_0}=\lambda^2+\sin^2r\parallel v\parallel^2_{\mathsf{T}_\phi\M}$$
for $(v,\lambda)\in\mathsf{T}_{(\phi,r)}\Sigma_0$ with $v\in\mathsf{T}_\phi\M$ and $\lambda\in\R$. Furthermore, we have the equality
$$\mathsf{Ric}_{\Sigma_0}(v+\lambda\tfrac{\partial}{\partial r},v+\lambda\tfrac{\partial}{\partial r})=\mathsf{Ric}_\M(v,v)+(1-n\cos^2r)\parallel v\parallel^2_{\mathsf{T}_\phi\M}+n\lambda^2$$
for $(v,\lambda)\in\mathsf{T}_{(\phi,r)}\Sigma_0$. In particular, $\mathsf{Ric}_{\Sigma_0}\geq n$ if and only if $\mathsf{Ric}_\M\geq n-1$.
\end{lemma}
\
\\
For fixed $\varepsilon>0$ we embed $\Sigma(\M)\setminus(B_\varepsilon(\mathcal{S})\cup B_\varepsilon(\mathcal{N}))$ in a complete Riemannian manifold $\tilde\M_\varepsilon$ whose Ricci curvature is bounded from below. This inclusion $\Sigma(\M)\setminus(B_\varepsilon(\mathcal{S})\cup B_\varepsilon(\mathcal{N}))\subseteq\tilde\M_\varepsilon$ implies that $\mu^\varepsilon_s$ is the unique $s$-intermediate point of $\mu^\varepsilon_0$ and $\mu^\varepsilon_1$ and satisfies
\begin{equation*}
\mathsf{S}_{n'}(\mu^\varepsilon_s|\nu)\leq\tau^{(1-s)}_{n-1,n'}(\theta)\mathsf{S}_{n'}(\mu^\varepsilon_0|\nu)+
\tau^{(s)}_{n-1,n'}(\theta)\mathsf{S}_{n'}(\mu^\varepsilon_1|\nu),
\end{equation*}
where
$$\theta:=\underset{x_1\in\mathsf{supp}(\mu_1)}{\underset{x_0\in\mathsf{supp}(\mu_0),}{\inf}}\mathsf{d}_\Sigma(x_0,x_1)$$
for all $\varepsilon>0$. Passing to the limit $\varepsilon\to 0$ yields according to the convergence statements,
\begin{equation*}
\mathsf{S}_{n'}(\mu_s|\nu)\leq\tau^{(1-s)}_{n-1,n'}(\theta)\mathsf{S}_{n'}(\mu_0|\nu)+
\tau^{(s)}_{n-1,n'}(\theta)\mathsf{S}_{n'}(\mu_1|\nu)
\end{equation*}
for all $n'\geq n+1$.
\end{proof}

Because of Theorem \ref{sphconriemman} we can apply the Lichnerowicz theorem \cite{lva} in order to obtain a lower bound on the spectral gap of the Laplacian on the spherical cone:

\begin{corollary}[Lichnerowicz estimate, Poincar\'e inequality]
Let $(\Sigma(\M),\mathsf{d}_\Sigma,\nu)$ be the $n$-spherical cone of a compact and complete $n$-dimensional Riemannian manifold $(\mathsf{M},\mathsf{d}, \V)$ with $\mathsf{Ric}\geq n-1$. Then for every $f\in\mathsf{Lip}(\Sigma(\M))$ fulfilling $\int_{\Sigma(\M)} f \ d\nu=0$ the following inequality holds true,
\begin{equation*}
\int_{\Sigma(\M)} f^2d\nu\leq \tfrac{1}{n+1}\int_{\Sigma(\M)}|\nabla f|^2d\nu.
\end{equation*}
\end{corollary}

The Lichnerowicz estimate implies that the Laplacian $\Delta$ on the spherical cone $(\Sigma(\M),\mathsf{d}_\Sigma,\nu)$ defined by the identity
$$\int_{\Sigma(\M)} f\cdot\Delta g \ d\nu=-\int_{\Sigma(\M)}\nabla f\cdot\nabla g \ d\nu$$ admits a spectral gap $\lambda_1$ of size at least $n+1$,
$$\lambda_1\geq n+1.$$

\bigskip
\
\\
\textbf{Acknowledgement.} The second author would like to thank Jeff Cheeger for stimulating discussions during a visit at Courant Institute in 2004, in particular, for posing the problem treated in Theorem \ref{Cheeger}, as well as for valuable comments on an early draft of this paper.

\end{document}